\documentclass[a4paper,draft]{amsproc}
\usepackage{amssymb}
\usepackage{amscd} 
\theoremstyle{plain}
\newtheorem*{theoA}{Theorem A}
\newtheorem*{theoB}{Theorem B}
\newtheorem*{theoC}{Theorem C}
\newtheorem*{theoD}{Theorem D}
\newtheorem*{theoE}{Theorem E}
\newtheorem*{theoF}{Theorem F}
\newtheorem*{theoG}{Theorem G}
\newtheorem*{theoH}{Theorem H}
\newtheorem*{theoI}{Theorem I}
 \newtheorem{theo}{Theorem}[section]
 
 \newtheorem{lem}{Lemma}[section]
 
\theoremstyle{definition}
 \newtheorem{exm}{Example}[section]
 \newtheorem{ques}{Question}[section]
 
 \newtheorem{note}{Note}[section]
\theoremstyle{remark}
 \newtheorem{rem}{Remark}[section]
 \newcommand{\ol}{\overline}
\newcommand{\be}{\begin{equation}}
\newcommand{\ee}{\end{equation}}
\newcommand{\beas}{\begin{eqnarray*}}
\newcommand{\eeas}{\end{eqnarray*}}
\newcommand{\bea}{\begin{eqnarray}}
\newcommand{\eea}{\end{eqnarray}}

\numberwithin{equation}{section}
\renewcommand{\leq}{\leqslant}
\renewcommand{\geq}{\geqslant}

\title[On uniqueness of two meromorphic functions ...]{\LARGE On uniqueness of two meromorphic functions sharing a small function}

\numberwithin {equation}{section}
\date{}

\author[Molla Basir Ahamed]{Molla Basir Ahamed}

\address{ Department of Mathematics, Kalipada Ghosh Tarai Mahavidyalaya, West Bengal 734014, India.}
\email{basir\_math\_kgtm@yahoo.com, bsrhmd117@gmail.com.}

\begin{document}

\vspace{18mm} \setcounter{page}{1} \thispagestyle{empty}


\begin{abstract}
In this paper, we have investigated the uniqueness problems of entire and meromorphic functions concerning differential polynomials sharing a small function. Our results radically extended and improved the results of \emph{Bhoosnurmath-Pujari} \cite{Bho & Puj-IJA-2013} and \emph{Harina - Anand} \cite{Wag & Ana-AM-2016} not only by sharing small function instead of fixed point but also reducing the lower bound of $ n $. The authors \emph{Harina-Anand} \cite{Wag & Ana-AM-2016} made plenty of mistakes in their paper. We have corrected all of them in a more convenient way. At last some open questions have been posed for further study in this direction.
\end{abstract}

\maketitle

\section{Introduction, definitions and main results}
The Nevanlinna theory mainly describes the asymptotic distribution of solutions of the equation $f(z)=w$, as $w$ varies. At the outset, we assume that readers are familiar with the basic Nevanlinna Theory \cite{Hay-CP-1964}. First we explain the general sharing notion. Let $f$ and $g$ be two non-constant meromorphic functions in the complex plane $\mathbb{C}$. Two meromorphic functions $f$ and $g$ are said to share a value $w\in\mathbb{C}\cup\{\infty\}$ $ IM $ (ignoring multiplicities) if $ f $ and $ g $ have the same $ w $-points counted with ignoring multiplicities. If multiplicities of $ w $-points are counted, then $ f $ and $ g $ are said to share $ w $ $ CM $ (counting multiplicities).\par
When $w=\infty$ the zeros of $f-w$ means the poles of $f$.\par
It is well known that if two moromorphic functions $ f $ and $ g $ share four distinct values $ CM $, then one is \emph{M{\"o}bius Transformation} of the other. In $ 1993 $, corresponding to one famous question of \emph{Hayman} \cite{Hay-AP-1967}, \emph{Yang-Hua} \cite{Yan & Hua-AASFM-1997} showed that similar conclusions hold for certain types of differential polynomials when they share only one value.\par Recently by using the same argument as in \cite{Yan & Hua-AASFM-1997}, \emph{Fang-Hong} \cite{Fan & Hon-IJPAM-2001} the following result was obtained.
\begin{theoA}
	Let $ f $ and $ g $ be two transcendental entire functions, $ n\geq 11 $, an integer. If $ f^n(f-1)f^{\prime} $ and $ g^n(g-1)g^{\prime} $ share $ 1 $ $ CM $, then $ f\equiv g $.
\end{theoA}\par The following example shows that in \emph{Theorem A} one simply can not replace ``entire" by ``meromorphic " functions.
\begin{exm}
	Let \beas f(z)=\frac{(n+2)}{(n+1)}\frac{e^z+\ldots+e^{(n+1)z}}{1+e^z+\ldots+e^{(n+1)z}} \eeas and \beas f(z)=\frac{(n+2)}{(n+1)}\frac{1+e^z+\ldots+e^{nz}}{1+e^z+\ldots+e^{(n+1)z}}. \eeas\par It is clear that $ f(z)=e^zg(z) $. Also $ f^n(f-1)f^{\prime} $ and $ g^n(g-1)g^{\prime} $ share $ 1 $ $ CM $ but note that $ f\not\equiv g $.
\end{exm}

In $ 2004 $, \emph{Lin-Yi} \cite{Lin & Yi-CVTA-2004} extended Theorem A and obtained the following results.
\begin{theoB}\cite{Lin & Yi-CVTA-2004}
	Let $ f $ and $ g $ be two transcendental entire functions, $ n\geq 7 $ an integer. If $ f^n(f-1)f^{\prime} $ and $ g^n(g-1)g^{\prime} $ share $ z $ $ CM $, then $ f\equiv g $.
\end{theoB}
\begin{theoC}\cite{Lin & Yi-CVTA-2004}
	Let $ f $ and $ g $ be two transcendental meromorphic functions, $ n\geq 12 $ an integer. If $ f^n(f-1)f^{\prime} $ and $ g^n(g-1)g^{\prime} $ share $ z $ $ CM $, then either $ f\equiv g $ or \beas g=\frac{(n+2)(1-h^{n+1})}{(n+1)(1-h^{n+2})},\;\;\;\; f=\frac{(n+2)h(1-h^{n+1})}{(n+1)(1-h^{n+2})},  \eeas where $ h $ is a non-constant meromorphic function.
\end{theoC}
\begin{theoD}\cite{Lin & Yi-CVTA-2004}
	Let $ f $ and $ g $ be two transcendental meromorphic functions, $ n\geq 13 $ an integer. If $ f^n(f-1)^2f^{\prime} $ and $  g^n(g-1)^2g^{\prime} $ share $ z $ $ CM $, then $ f\equiv g $.
\end{theoD}
To improve all the above mentioned results, natural questions arise as follows.
\begin{ques}
	Keeping all other conditions intact, is it possible to reduce further  the lower bounds of $ n $ in the above results ?
\end{ques}
\begin{ques}
	Is it also possible to replace the transcendental meromorphic (entire) functions by a more general class of meromorphic (entire) functions in all the above mentioned results ?
\end{ques}\par 
In $ 2013 $, \emph{Bhoosnurmath-Pujari} \cite{Bho & Puj-IJA-2013}, answered the above questions affirmatively and obtained the following results.

\begin{theoE}\cite{Bho & Puj-IJA-2013}
	Let $ f $ and $ g $ be two non-constant meromorphic functions, $ n\geq 11 $ be an integer. If $ f^n(f-1)f^{\prime} $ and $ g^n(g-1)g^{\prime} $ share $ z $ $ CM $, $ f $ and $ g $ share $ \infty $ $ IM $, then either $ f\equiv g $ or \beas    g=\frac{(n+2)(1-h^{n+1})}{(n+1)(1-h^{n+2})},\;\;\;\; f=\frac{(n+2)h(1-h^{n+1})}{(n+1)(1-h^{n+2})},  \eeas where $ h $ is a non-constant meromorphic function.
\end{theoE}
\begin{theoF}\cite{Bho & Puj-IJA-2013}
	Let $ f $ and $ g $ be two non-constant meromorphic functions, $ n\geq 12 $ an integer. If $ f^n(f-1)^2f^{\prime} $ and $ g^n(g-1)^2g^{\prime} $ share $ z $ $ CM $, $ f $ and $ g $ share $ \infty $ $ IM $, then $ f\equiv g $.
\end{theoF}
\begin{theoG}\cite{Bho & Puj-IJA-2013}
	Let $ f $ and $ g $ be two non-constant entire functions, $ n\geq 7 $ be an integer. If $ f^n(f-1)f^{\prime} $ and $ g^n(g-1)g^{\prime} $ share $ z $ $ CM $, then $ f\equiv g $.
\end{theoG}
In this direction, for the purpose of extension Theorem E and F, one may ask the following question.
\begin{ques}\label{q1.3}
	Keeping all other conditions intact in Theorem E, F and G, is it possible to replace respectively $ f^n(f-1)f^{\prime} $ and  $ g^n(g-1)g^{\prime} $ by $ f^n(f-1)^mf^{\prime} $ and $ g^n(g-1)^mg^{\prime} $ ?
\end{ques}
\par Next the following question is inevitable.
\begin{ques}\label{q1.4}
	Is it possible to omit the second conclusions of Theorems C and E ?
\end{ques}
In $ 2016 $, \emph{Waghmore-Anand} \cite{Wag & Ana-AM-2016} answer the \emph{Questions \ref{q1.3}} and \emph{\ref{q1.4}}\; affirmatively and obtained the following results.

\begin{theoH}\cite{Wag & Ana-AM-2016}
	Let $ f $ and $ g $ be two non-constant meromorphic functions, $ n\geq m+10 $ be an integer. If $ f^n(f-1)^mf^{\prime} $ and $ g^n(g-1)^mg^{\prime} $ share $ z $ $ CM $, $ f $ and $ g $ share $ \infty $ $ IM $, then $ f\equiv g $.
\end{theoH}
\begin{theoI}\cite{Wag & Ana-AM-2016}
	Let $ f $ and $ g $ be any two non-constant entire functions, $ n\geq m+6 $ an integer. If $ f^n(f-1)^mf^{\prime} $ and $ g^n(g-1)^mg^{\prime} $ share $ z $ $ CM $, then $ f\equiv g $.
\end{theoI}
\begin{note}
	We see that in the results of \emph{Waghmore - Anand}, for $ m=2 $, \emph{Theorem H} reduces to \emph{Theorem F } and for $ m=1 $,  \emph{Theorem I} reduces to \emph{Theorem G }. 
\end{note}
\par 
\begin{rem}
	We notice that in the proof of Theorem H and hence in the case of Theorem I also, there are plenty of mistakes made by the authors \emph{Waghmore-Anand} \cite{Wag & Ana-AM-2016}. We mention below few of them.
	\begin{enumerate}
		\item[(i)] In \cite[page-947]{Wag & Ana-AM-2016}, just before \emph{Case 2}, the authors obtained that the coefficient of $ T(r,g) $ is $ (n-m-8) $, while actually it will be $ (n+m-8) $.
		\item[(ii)]  In \cite[page-948]{Wag & Ana-AM-2016}, just before \emph{Case 3}, the authors finally obtained that ``$ h^{n+m}-1=0 $, $ h^{n+1}-1=0 $, which imply $ h=1 $". Note that this possible only when $ gcd(n+m, n+1)=1 $ but which is not true if one consider some suitable value of $ n $ and $ m $. For example if we choose $ n=3 $ and $ m=5 $, we note that $ gcd(n+m, n+1)=gcd(8,4)=4\neq 1 $.
		\item[(iii)] We observe that in \cite[equation (49), page-950]{Wag & Ana-AM-2016}, the coefficient of $ T(r,g) $ is $ \displaystyle\frac{m}{n+m-1} $ while actually it should be $ \displaystyle\frac{m}{n+m+1} $. 
	\end{enumerate}
\end{rem}
\par Thus we see that study on derivative or differential polynomial has a long history, several authors have been main engaged to find a possible relationship or certain forms of a function $ f $ when it shares small function with its derivative or differential polynomial (see \cite{Aha & CKMS & 2018} - \cite{Ban & Aha & JCA & 2019}.) 
\par In this paper, our aim is to correct all the mistakes made by \emph{Waghmore-Anand} \cite{Wag & Ana-AM-2016} and at the same time to get an improved and extended version results of all the above mentioned Theorems A - I.\par To this end, throughout the paper, we will use the following transformations.
Let \beas\mathcal{P}(w)=w^{n+m}+\ldots+a_nw^n+\ldots+a_0=a_{n+m}\displaystyle\prod_{i=1}^{s}(w-w_{p_i})^{p_i}\eeas where $a_j(j=0,1,2,\ldots,n+m-1)$ and $w_{p_i}(i=1,2,...,s)$ are distinct finite complex numbers and $2\leq s\leq n+m$ and $p_1,p_2,\ldots,p_s$, $s\geq 2$, $n$, $m$ and $k$ are all positive integers with $\displaystyle\sum_{i=1}^{s}p_i=n+m$. Also let $p>\displaystyle\max_{p\neq p_i,i=1,\ldots,r}\{p_i\}$, $r=s-1$, where $s$ and $r$ are two positive integers.\\

Let $\mathcal{Q}(w_*)=\displaystyle\prod_{i=1}^{s-1}(w_*+w_p-w_{p_i})^{p_i}=b_qw_*^q+b_{q-1}w_*^{q-1}+\ldots+b_0$, where $ w_*=w-w_p$, $q=n+m-p$. So it is clear that $ \mathcal{P}(w)=w_*^p\mathcal{Q}(w_*) $\par In particular, if we choose $ b_i=\displaystyle(-1)^{i}\;{^qC_i}$, for $ i=0,1,\ldots,q $. Then we get, easily $\mathcal{P}_*(w)=w_*^p(w_*-1)^q$.\par
Note that if $ w_p=0 $ and $ p=n $, then we get $ w=w_* $ and $ \mathcal{P}_*(w)=w^n(w-1)^m $.
\par Observing all the above mentioned results, we note that $ h^n(h-1)h^{\prime} $ or $ h^n(h-1)^2h^{\prime} $ $ (h=f \;\text{or}\; g) $ are a special form of  $ h^n(h-1)^mh^{\prime} $, $ m\geq 1 $ be an integer.\\ \par 
So for the improvements and extensions of the above mentioned results further to a large extent, the following questions are inevitable.
\begin{ques}\label{q1.5}
	Is it possible to replace $ f^n(f-1)^mf^{\prime} $ and $ g^n(g-1)^mg^{\prime} $ by a more general expressions of the form $ \mathcal{P}_*(f)f_*^{\prime
	}=f_*^p(f_*-1)^qf_*^{\prime} $ and $ \mathcal{P}_*(g)g_*^{\prime
	}=g_*^p(g_*-1)^qg_*^{\prime} $ respectively in all the above mentioned results ?
\end{ques} 
\par If the answer of the \emph{Question \ref{q1.5}}\; is found to be  affirmative, then one my ask the following questions.
\begin{ques}\label{q1.6}
	Is it possible to reduce further the lower bounds of $ n $ in Theorems  $ E $, $ F $, $ G $  and $ H $ ? 
\end{ques}
\begin{ques}\label{q1.7}
	Is it also possible to replace  sharing $ z $  $ CM $ by sharing $ \alpha(z) $ $ CM $ in Theorem G and H ?
\end{ques}
\par Answering all the above mentioned questions affirmatively is the main motivation of writing this paper. \par
Following two theorems are the main results of this paper improving and extending all the above mentioned results to a more convenient way and compact form.
\begin{theo}\label{th.1}
	Let $ f $ and $ g $ hence $ f_*=f-w_p $ and $ g_*=g-w_p $, $ w_p\in\mathbb{C} $ be any two non-constant non- entire meromorphic functions, $ n\geq q+9 $, $ q $ $ \in\mathbb{N} $, be an integer. If $ \mathcal{P}_*(f)f_*^{\prime}=f_*^p(f_*-1)^qf_*^{\prime} $ and $ \mathcal{P}_*(g)g_*^{\prime}=g_*^p(g_*-1)^qg_*^{\prime} $ share $ \alpha\equiv\alpha(z) $ $ (\not\equiv 0, \infty) $ $ CM $, $ f_* $ and $ g_* $ share $ \infty $ $ IM $, then $ f\equiv g $.
\end{theo}
\begin{theo}\label{th.2}
	Let $ f $ and $ g $ hence $ f_*=f-w_p $ and $ g_*=g-w_p $, $ w_p\in\mathbb{C} $ be any two non-constant entire functions, $ n\geq q+5 $, $ q $ $ \in\mathbb{N} $, be an integer. If $ \mathcal{P}_*(f)f_*^{\prime}=f_*^p(f_*-1)^qf_*^{\prime} $ and $ \mathcal{P}_*(g)g_*^{\prime}=g_*^p(g_*-1)^qg_*^{\prime} $ share $ \alpha\equiv\alpha(z) $ $ (\not\equiv 0, \infty) $ $ CM $, then $ f\equiv g $.
\end{theo}
\section{Some lemmas}
In this section we present some lemmas which will be needed in sequel.
\begin{lem}\label{l2.1}\cite{Xu & Lu & Yi-CMA-2010}
	Let $ f_1 $, $ f_2 $ and $ f_3 $ be non constant meromorphic functions such that $ f_1+f_2+f_3=1 $. If $ f_1 $, $ f_2 $ and $ f_3 $ are linearly independent, then \beas  T(r,f_1)<\sum_{i=1}^{3}N_2\left(r,\frac{1}{f_i}\right)+\sum_{i=1}^{3}\ol N(r,f)+o(T(r)), \eeas where $ T(r)=\displaystyle\max_{1\leq i\leq 3}\bigg\{T(r,f_i)\bigg\} $ and $ r\not\in E $.
\end{lem} 
\begin{lem}\label{l2.2} \cite{Yan & Yi-MA-2003}
	Let $ f_1 $ and $ f_2 $ be two non-constant meromorphic functions. If $ c_1f_1+c_2f_2=c_3 $, where $ c_i $, $i=1,2,3$ are non-zero constants, then \beas  T(r,f_1)\leq \ol N(r,f_1)+\ol N\left(r,\frac{1}{f_1}\right)+\ol N\left(r,\frac{1}{f_2}\right)+S(r,f_1). \eeas
\end{lem}
\begin{lem}\label{l2.3}\cite{Yan & Yi-MA-2003}
	Let $ f $ be a non-constant meromorphic function and $ k $ be a non-negative integer, then \beas  N\left(r,\frac{1}{f^{(k)}}\right)\leq N(r,\frac{1}{f})+k\ol N(r,f)+S(r,f).  \eeas
\end{lem}
\begin{lem}\label{l2.4}\cite{Yan-MZ-1972}
	Suppose that $ f $ is a non-constant meromorphic function and $ P(f)=a_nf^n+a_{n-1}f^{n-1}+\ldots+a_1f+a_0 $, where $ a_n(\not\equiv  0) $, $ a_{n-1},\ldots, a_1, a_0 $ are small meromorphic functions of $ f(z) $. Then \beas  T(r,P(f))=n\;T(r,f)+S(r,f). \eeas
\end{lem}
\begin{lem}\label{l2.5}\cite{Yi-KMJ-1990}
	Let $ f_1 $, $ f_2 $ and $ f_3 $ be three meromorphic functions satisfying $ \displaystyle\sum_{i=1}^{3}f_i=1, $ then the functions $ g_1=-\displaystyle\frac{f_1}{f_2} $, $ g_2=\displaystyle\frac{1}{f_2} $ and $ g_3=-\displaystyle\frac{f_1}{f_2} $ are linearly independent when $ f_1 $, $ f_2 $ and $ f_3 $ are linearly independent.
\end{lem}
\begin{lem}\label{l2.7}
	Let $ f $ and $ g $ and hence $ f_*=f-w_p $ and $ g_*=g-w_p $ be two non-constant meromorphic functions and $  \alpha\equiv\alpha (z)  $ $ (\not\equiv 0, \infty) $ be a small function of $ f $ and $ g $. If $ \mathcal{P}_*(f)f_*^{\prime}=f_*^p(f_*-1)^qf_*^{\prime} $ and $ \mathcal{P}_*(g)g_*^{\prime}=g_*^p(g_*-1)^qg_*^{\prime} $ share $ \alpha $  $ CM $ and $ p\geq 7 $, then \beas T(r,g_*)\leq\left(\frac{p+q+2}{p-6}\right)T(r,f_*)+S(r,g_*)  \eeas
\end{lem}
\begin{proof}
	Applying \emph{Second Fundamental Theorem} on $ \mathcal{P}_*(g)g_*^{\prime} $, we get \bea\label{e2.1} && T\left(r,\mathcal{P}_*(g)g_*^{\prime}\right)\\ &\leq& \ol N(r,\mathcal{P}_*(g)g_*^{\prime})+\ol N\left(r,\frac{1}{\mathcal{P}_*(g)g_*^{\prime}}\right)+\ol N\left(r,\frac{1}{\mathcal{P}_*(g)g_*^{\prime}-\alpha}\right)+S(r,g_*)\nonumber\\ &\leq & \ol N\left(r,\frac{1}{g_*^p(g_*-1)^qg_*^{\prime}}\right)+\ol N(r,g_*)+\ol N\left(r,\frac{1}{g_*^p(g_*-1)^qg_*^{\prime}-\alpha}\right)+S(r,g_*)\nonumber\eea
	Next by applying \emph{First fundamental Theorem}, \bea &&\label{e2.2} (p+q)T(r,g)\\\nonumber &\leq& T(r,g_*^p(g_*-1)^q)+S(r,g_*)\\ &\leq&\nonumber T(r,g_*^p(g_*-1)^qg_*^{\prime})+T\left(r,\frac{1}{g_*^{\prime}}\right)+S(r,g_*)\nonumber .\eea After combining (\ref{e2.1}) and (\ref{e2.2}), we get \bea\label{e2.3} && (p+q)T(r,g)\\ &\leq&\nonumber\ol N\left(r,\frac{1}{g_*}\right)+\ol N(r,0;g_*-1)+\ol N(r,g_*)+\ol N\left(r,\frac{1}{g_*^{\prime}}\right)+\ol N\left(r,\frac{1}{f_*^p(f_*-1)^qf_*^{\prime}-\alpha}\right)\\ &&+S(r,g_*)+T(r,g_*^{\prime})\nonumber.\eea
	Again since $ S(r,g_*)=T(r,\alpha)=S(r,f_*) $, so we must have \bea\label{e2.4} &&\ol N\left(r,\frac{1}{f_*^p(f_*-1)^qf_*^{\prime}-\alpha}\right)\\ &\leq&\nonumber T\left(r,\alpha;f_*^p(f_*-1)^qf_*^{\prime}\right)+O(1)\\ &\leq&\nonumber T(r,f_*^p)+T(r,(f_*-1)^q+T(r,f_*^{\prime})+T(r,\alpha)+O(1)\\ &\leq&\nonumber p\; T(r,f_*)+q\;T(r,f_)+2T(r,f_*)+S(r,g_*)\\ &=&\nonumber (p+q+2)T(r,f_*)+S(r,g_*).   \eea
	By using (\ref{e2.4}) in (\ref{e2.3}), we get \beas && (p+q)T(r,g_*)\\ &\leq& (q+6)T(r,g_*)+(p+q+2)T(r,f_*)+S(r,g_*).\eeas i.e., \beas T(r,g)\leq\left(\frac{p+q+2}{p-6}\right)T(r,f)+S(r,g),  \eeas where $ p\geq 7. $
\end{proof}
\begin{lem}\label{l2.8}
	Let $ f $ and $ g $ and hence $ f_*=f-w_p $ and $ g_*=g-w_p $ be two non-constant entire functions and $  \alpha\equiv\alpha (z)  $ $ (\not\equiv 0, \infty) $ be a small function of $ f $ and $ g $. If $ \mathcal{P}_*(f)f_*^{\prime}=f_*^p(f_*-1)^qf_*^{\prime} $ and $ \mathcal{P}_*(g)g_*^{\prime}=g_*^p(g_*-1)^qg_*^{\prime} $ share $ \alpha $  $ CM $ and $ p\geq 5 $, then \beas T(r,g_*)\leq\left(\frac{p+q+2}{p-3}\right)T(r,f_*)+S(r,g_*)  \eeas
\end{lem}
\begin{proof}
	Since $ f $ and $ g $ both are entire functions, so we must have $ \ol N(r,f)=0=\ol N(r,g) $.\par Proceeding exactly as in the line of the proof of Lemma \ref{l2.7}, we can prove the lemma.
\end{proof}
\begin{lem}\label{lm9} Let $ \Psi(z)=c^2(z^{p-q}-1)^2-4b(z^{p-2q}-1)(z^{p}-1)$ , where $ b,c\in\mathbb{C}-\{0\}, \displaystyle\frac{c^2}{4b}=\displaystyle\frac{p(p-2q)}{(p-q)^2}\neq 1,\;$, then $ \Psi(z) $ has exactly one multiple zero of multiplicity $ 4 $ which is $ 1 $.
\end{lem}

\begin{proof}
	We claim that $ \Psi(1)=0 $ with multiplicity $ 4 $ and all other zeros of $ \Psi(w) $ are simple. Let $ F(t)=\displaystyle\frac{1}{2}\Psi(e^t)e^{(q-p)t}. $ Then\beas && F(t)\\ &=&\displaystyle\frac{1}{2}\bigg\{4b(1-e^{pt})(1-e^{(p-2q)t})-c^2(1-e^{(p-q)t})\bigg\}e^{(q-p)t}\\ &=&(4b-c^2)\cosh (q-p)t-4b\cosh qt+c^2.\eeas Next we see that for $ t=0 $, $ F(t)=0 $, $ [F(t)]^{\prime}=0 $, $ [F(t)]^{\prime\prime}=0 $ since $ \displaystyle\frac{c^2}{4b}=\frac{p(p-2q)}{(p-q)^2} $ and $ [F(t)]^{\prime\prime\prime}=0 $ but $ [F(t)]^{(iv)}\neq 0 $ where $$ [F(t)]^{\prime}=(4b-c^2)(q-p)\sinh (q-p)t-4bq\sinh qt, $$ $$ [F(t)]^{\prime\prime}=(4b-c^2)(q-p)^2\cosh (q-p)t-4bq^2\cosh qt, $$ $$ [F(t)]^{\prime\prime\prime}=(4b-c^2)(q-p)^3\sinh (q-p)t-4bq^3\sinh qt $$ and $$ [F(t)]^{(iv)}=(4b-c^2)(q-p)^4\cosh (q-p)t-4bq^4\cosh qt.$$ Therefore it is clear that $ F(0)=0 $ with multiplicity $ 4 $ and hence $ \Psi(1)=0 $ with multiplicity $ 4 $.
	\par  Next we suppose that $ \Psi(w)=0=\Psi^{\prime}(w) $, for some $ w\in\mathbb{C} $. Then $ F(t)=0=F^{\prime}(t) $ for every $ t $ satisfying $ e^{qt}=w $. Now from $ F(t)=0 $ and $ F^{\prime}(t)=0 $, we obtained respectively \bea\label{e2.3} (4b-c^2)\cosh (q-p)t-4b\cosh qt+c^2=0  \eea and \bea\label{e2.4} (4b-c^2)(q-p)\sinh (q-p)t-4qb\sinh qt=0. \eea  Since $ \cosh^2(q-p)t-\sinh^2(q-p)t=1 $, so from (\ref{e2.3})  and (\ref{e2.4}), we get \beas \displaystyle\frac{(4b\cosh qt-c^2)^2}{(4b-c^2)^2}-\displaystyle\frac{16q^2b^2\sinh^2 qt}{(4b-c^2)^2(q-p)^2}=1. \eeas i.e., \beas (q-p)^2(4b\cosh^2qt-c^2)^2-16q^2b^2(\cosh^2qt-1)=(4b-c^2)^2(q-p)^2.\eeas i.e., \bea\label{e2.5}  \bigg\{\cosh qt-1\bigg\}\bigg\{\cosh qt-\frac{a^2(q-p)^2}{2bp(p-2q)}+1\bigg\} =0. \eea \par Since $ \displaystyle\frac{c^2}{4b}=\displaystyle\frac{p(p-2q)}{(p-q)^2} $, then $ \displaystyle\frac{c^2(q-p)^2}{2bq(q-2p)}=2 $, so we see that the equation (\ref{e2.5}) reduces to $ \bigg\{\cosh qt-1\bigg\}^2=0. $ i.e., we get $ e^{qt}=1=w.$
\end{proof}
\section{Proofs of the theorems}
\begin{proof}[\bf{Proof of Theorem \ref{th.1}}] Since $ \mathcal{P}_*(f)f_*^{\prime} $ and $ \mathcal{P}_*(g)g_*^{\prime} $ share $ \alpha\equiv\alpha(z) $ $ CM $, $ f $ and $ g $ share $ \infty $ $ IM $, so we suppose that \bea\label{e3.1} \mathcal{H}\equiv\frac{\mathcal{P}_*(f)f_*^{\prime}-\alpha}{\mathcal{P}_*(g)g_*^{\prime}-\alpha}\equiv\frac{f_*^p(f_*-1)^qf_*^{\prime}-\alpha}{g_*^p(g_*-1)^qg_*^{\prime}-\alpha}.  \eea
	Then from (\ref{e2.4}) and (\ref{e3.1}), we get \beas && T(r,\mathcal{H})\\ &=&T\left(r,\frac{\mathcal{P}_*(f)f_*^{\prime}-\alpha}{\mathcal{P}_*(g)g_*^{\prime}-\alpha}\right)\\ &\leq& T(\mathcal{P}_*(f)f_*^{\prime}-\alpha)+T(r,\mathcal{P}_*(g)g_*^{\prime}-\alpha)+O(1)\\ &\leq& T(r,f_*^p(f_*-1)^qf_*^{\prime}-\alpha)+T(r,g_*^p(g_*-1)^qg_*^{\prime}-\alpha)+O(1)\\ &\leq&(p+q+2) (T(r,f_*)+T(r,g_*))+S(r,f_*)+S(r,g_*)\\ &\leq& 2(p+q+2)T_*(r)+S_*(r), \eeas where $ T_*(r)=\max\{T(r,f_*),T(r,g_*)\} $ and $ S_*(r)=\max\{S(r,f_*),S(r,g_*)\} $.\\
	i.e., \bea\label{e3.2} T(r,\mathcal{H})=O(T_*(r)).  \eea Again from (\ref{e3.1}), we  see that the zeros and poles of $ \mathcal{H} $ are multiple and hence \bea\label{e3.3}  \ol N(r,\mathcal{H})\leq \ol N_L(r,f),\;\;\;\; \ol N\left(r,\frac{1}{\mathcal{H}}\right)\leq \ol N_L(r,g). \eea\par Let $ f_1=\displaystyle\frac{f_*^p(f_*-1)^qf_*^{\prime}}{\alpha} $,\;\;\;\; $ f_2=\mathcal{H} $\;\;\;\; and \;\;\;\;$ f_3=-\mathcal{H}\displaystyle\frac{g_*^p(g_*-1)^qg_*^{\prime}}{\alpha} $.\par  Thus we get $ f_1+f_2+f_3=1 $. Next we denote $ T(r)=\max\{T(r,f_1),T(r,f_2),T(r,f_3)\}. $\par We have, \beas T(r,f_1)=O(T(r,f_*))\eeas  \beas T(r,f_2)=O(T(r,f_*)+T(r,g_*))=T(r,f_3). \eeas So we have $ T(r,f_i)=O(T_*(r)) $ for $ i=1,2,3 $ and hence $ S(r,f_*)+S(r,g_*) =o(T_*(r))$.\\
	
	Next we discuss the following cases.\\
	\noindent{\bfseries{Case 1.}} Suppose none of $ f_2 $ and $ f_3 $ is a constant. If $ f_1 $, $ f_2 $ and $ f_3 $ are linearly independent, then by Lemma \ref{l2.1} and \ref{l2.4}, we have \bea\label{e3.4} && T(r,f_1)\\ &\leq&\sum_{i=1}^{3}N_2\left(r,\frac{1}{f_i}\right)+\sum_{i=1}^{3}\ol N(r,f_i)+o(T(r))\nonumber\\ &\leq&\nonumber N_2\left(r,\frac{\alpha}{f_*^p(f_*-1)^qf_*^{\prime}}\right)+N_2\left(r,\frac{1}{\mathcal{H}}\right)+N_2\left(r,\frac{\alpha}{\mathcal{H}g_*^p(g_*-1)^qg_*^{\prime}}\right)\\ &&+\ol N(r,f_*^p(f_*-1)^qf_*^{\prime})\nonumber+\ol N(r,\mathcal{H})+\ol N(r,\mathcal{H}g_*^p(g_*-1)^qg_*^{\prime})+o(T(r))\nonumber\\ &\leq& N_2\left(r,\frac{1}{f_*^p(f_*-1)^qf_*^{\prime}}\right)+2N_2\left(r,\frac{1}{\mathcal{H}}\right)+N_2\left(r,\frac{1}{g_*^p(g_*-1)^qg_*^{\prime}}\right)+\ol N(r,f_*)\nonumber\\ &&+2\ol N(r,\mathcal{H})+\ol N(r,g_*)+o(T(r)). \nonumber\eea We see that $ N_2\left(r,\displaystyle\frac{1}{\mathcal{H}}\right)\leq 2\ol N\left(r,\displaystyle\frac{1}{\mathcal{H}}\right)\leq 2\ol N_L(r,g_*) $,\;\;\; $ \ol N(r,\mathcal{H})\leq \ol N_L(r,f) $.\par Again since $ \ol N_L(r,f_*)=0=\ol N_L(r,g_*) $ and note that $ \ol N(r,f_*)=\ol N(r,g_*) $, so using all this facts, we get from (\ref{e3.4}) that  \bea\label{e3.5} && T(r,f_1)\\ &\leq& N_2\left(r,\frac{1}{f_*^p(f_*-1)^qf_*^{\prime}}\right)+N_2\left(r,\frac{1}{g_*^p(g_*-1)^qg_*^{\prime}}\right)+2\ol N(r,f_*)+o(T(r))\nonumber\\ &\leq&\nonumber N\left(r,\frac{1}{f_*^p(f_*-1)^qf_*^{\prime}}\right)-\bigg[N_{(3}\left(r,\frac{1}{f_*^p(f_*-1)^qf_*^{\prime}}\right)-2\ol N_{(3}\left(r,\frac{1}{f_*^p(f_*-1)^qf_*^{\prime}}\right)\bigg]\nonumber\\ &+&\nonumber N\left(r,\frac{1}{g_*^p(g_*-1)^qg_*^{\prime}}\right)-\bigg[N_{(3}\left(r,\frac{1}{g_*^p(g_*-1)^qg_*^{\prime}}\right)-2\ol N_{(3}\left(r,\frac{1}{g_*^p(g_*-1)^qg_*^{\prime}}\right)\bigg]\\ &&+2\ol N(r,f_*)+o(T(r)).\nonumber\eea \par Let $ z_0 $ be a zero of $ f_* $ of multiplicity $ r $, then $ z_0 $ is a zero of $ f_*^p(f_*-1)^qf_*^{\prime} $ of multiplicity $ pr+r-1\geq 3 $. Thus we have  \bea\label{e3.6} N_{(3}\left(r,\frac{1}{f_*^p(f_*-1)^qf_*^{\prime}}\right)-2\ol N_{(3}\left(r,\frac{1}{f_*^p(f_*-1)^qf_*^{\prime}}\right)\geq (p-2)N\left(r,\frac{1}{f_*}\right). \eea Similarly, we get \bea\label{e3.7} N_{(3}\left(r,\frac{1}{g_*^p(g_*-1)^qg_*^{\prime}}\right)-2\ol N_{(3}\left(r,\frac{1}{g_*^p(g_*-1)^qg_*^{\prime}}\right)\geq (p-2)N\left(r,\frac{1}{g_*}\right). \eea\par Let \beas \mathcal{F}=\displaystyle\frac{f_*^{p+q+1}}{p+q+1}-\frac{^qC_1}{p+q}f_*^{p+q}+\frac{^qC_2}{p+q-1}f_*^{p+q-1}+\ldots+(-1)^q\frac{1}{p+q}f_*^{p+1}\eeas and \beas \mathcal{G}=\displaystyle\frac{g_*^{p+q+1}}{p+q+1}-\frac{^qC_1}{p+q}g_*^{p+q}+\frac{^qC_2}{p+q-1}g_*^{p+q-1}+\ldots+(-1)^q\frac{1}{p+q}g_*^{p+1}.\eeas\par By Lemma \ref{l2.4}, we have \beas T(r,\mathcal{F})=(p+q+1)T(r,f_*)+S(r,f_*).\eeas\par It is clear $ \mathcal{F}^{\prime}=\alpha f_1 $. So we have \bea\label{e3.8} m\left(r,\frac{1}{\mathcal{F}}\right)\leq m\left(r,\frac{1}{\alpha f_1}\right)+m\left(r,\frac{\mathcal{F}^{\prime}}{\mathcal{F}}\right)\leq m\left(r,\frac{1}{f_1}\right)+S(r,f_*).  \eea\par By using \emph{First fundamental Theorem} and (\ref{e3.8}), we obtained \bea\label{e3.9} &&  T(r,\mathcal{F})\\ &=&\nonumber m\left(r,\frac{1}{\mathcal{F}}\right)+N\left(r,\frac{1}{\mathcal{F}}\right)\\ &\leq&\nonumber T(r,f_1)+N\left(r,\frac{1}{\mathcal{F}}\right)-N\left(r,\frac{1}{f_1}\right)+S(r,f_*)\\&\leq&T(r,f_1)+(p+1)N\left(r,\frac{1}{f_*}\right)+\sum_{i=1}^{q}N\left(r,\frac{1}{f_*-a_i}\right)-N\left(r,\frac{1}{f_1}\right)+S(r,f_*) \nonumber,\eea where $ a_i $ $ (i=1,2,\ldots,q) $ are the roots of the algebraic equation \beas  \frac{1}{p+q+1}z^q-\frac{^qC_1}{p+q}z^{q-1}+\frac{^qC_2}{p+q-1}z^{q-2}+\ldots+(-1)^q\frac{1}{p+1}=0. \eeas\par Using (\ref{e3.5}) - (\ref{e3.8}) in (\ref{e3.9}), we get \beas && T(r,\mathcal{F})\\ &\leq& N\left(r,\frac{1}{f_*^p(f_*-1)^qf_*^{\prime}}\right)+(2-p)N\left(r,\frac{1}{f_*}\right)+N\left(r,\frac{1}{g_*^p(g_*-1)^qg_*^{\prime}}\right)+(2-p)N\left(r,\frac{1}{g_*}\right)\\ &&+2\ol N(r,f_*)+(p+1)N\left(r,\frac{1}{f_*}\right)+\sum_{i=1}^{q}N\left(r,\frac{1}{f_*-a_i}\right)-N\left(r,\frac{1}{f_*^p(f_*-1)^qf_*^{\prime}}\right)+o(T(r)). \eeas i.e., \beas && (p+q+1)T(r,f_*)\\ &\leq& 3 N\left(r,\frac{1}{f_*}\right)+3 N\left(r,\frac{1}{g_*}\right)+\ol N(r,g_*)+q\;N\left(r,\frac{1}{g_*-1}\right)+2\ol N(r,f_*)\\ &&+\sum_{i=1}^{q}N\left(r,\frac{1}{f_*-a_i}\right)+o(T(r)) \\ &\leq&(q+5)T(r,f_*)+(q+4)T(r,g_*)+o(T(r)).   \eeas i.e., \bea\label{e3.10} (p-4)T(r,f_*)\leq (q+4)T(r,g_*)+o(T(r)).  \eea\\ \par 
	Let $ g_1=-\displaystyle\frac{f_3}{f_2}=\frac{g_*^p(g_*-1)^qg_*^{\prime}}{\alpha} $,\;\; $ g_2=\displaystyle\frac{1}{f_2}=\frac{1}{\mathcal{H}} $\;\; and\;\; $ g_3=-\displaystyle\frac{f_1}{f_2}=-\frac{f_*^p(f_*-1)^qf^{\prime}}{\alpha\mathcal{H}}. $\par Then we get $ g_1+g_2+g_3=1 $. By Lemma \ref{l2.5}, $ g_1 $, $ g_2 $ and $ g_3 $ are linearly independent since $ f_1 $, $ f_2 $ and $ f_3 $ are linearly independent. Proceeding exactly same way as done in above, we get  \bea\label{e3.11} (p-4)T(r,g_*)\leq (q+4)T(r,g_*)+o(T(r)).\eea\par Let $ T_*(r)=\max\{T(r,f_*),T(r,g_*)\} $. After combining (\ref{e3.10}) and (\ref{e3.11}), we get \beas (p-q-8)T_*(r)\leq o(T(r)),  \eeas which contradicts $ p\geq q+9. $\par Thus $ f_1 $, $ f_2 $ and $ f_3 $ must be linearly dependent. Therefore there exists three constants $ c_1 $, $ c_2 $ and $ c_3 $, at least one of them are non-zero such that \bea\label{e3.12} c_1f_1+c_2f_2+c_3f_3=0.  \eea
	\noindent{\bfseries{Subcase 1.1.}} If $ c_1=0 $, $ c_2\neq 0 $ and $ c_3\neq 0 $, then from (\ref{e3.12}) we get $ \displaystyle f_3=-\frac{c_2}{c_3}f_2 $ which implies $ g_*^p(g_*-1)^qg_*^{\prime}=\displaystyle\frac{c_2}{c_3}\alpha. $\par On integrating, we get \bea\label{e3.13}\frac{g_*^{p+q+1}}{p+q+1}-\frac{^qC_1\;g_*^{p+q}}{p+q}+\frac{^qC_2\;g_*^{p+q-1}}{p+q-1}\ldots+(-1)^q\frac{g_*^{p+1}}{p+1}=\frac{c_2}{c_3}\alpha+c,   \eea where $ c $ is an arbitrary constant.\par Thus we see that \beas T\left(r,\frac{g_*^{p+q+1}}{p+q+1}-\frac{^qC_1\;g_*^{p+q}}{p+q}+\frac{^qC_2\;g_*^{p+q-1}}{p+q-1}\ldots+(-1)^q\frac{g_*^{p+1}}{p+1}\right)\leq T(r,\alpha)+O(1).  \eeas i.e., \beas (p+q+1)T(r,g_*)\leq S(r,g_*).  \eeas\par Since $ p\geq q+9 $, so we get a contradiction.\\
	\noindent{\bfseries{Subcase 1.2.}} Let $ c_1\neq 0 $. Then from (\ref{e3.12}), we get \beas f_1= \left(-\frac{c_2}{c_1}\right)f_2+\left(-\frac{c_3}{c_1}\right)f_3. \eeas \par After substituting this in the relation $ f_1+f_2+f_3=1 $, we get \beas \left(1-\frac{c_2}{c_1}\right)f_2+\left(1-\frac{c_3}{c_1}\right)f_3=1, \eeas where $ (c_1-c_2)(c_1-c_3)\neq 0 $. So we get \bea\label{e3.14} \left(1-\frac{c_3}{c_1}\right)\frac{g_*^p(g_*-1)^qf_*^{\prime}}{\alpha}+\frac{1}{\mathcal{H}}=\left(1-\frac{c_2}{c_1}\right).  \eea\par Again we see that \beas && T(r,g_*^p(g_*-1)^qg_*^{\prime})\\ &\leq& T\left(r,\frac{g_*^p(g_*-1)^qg_*^{\prime}}{\alpha}\right)+T(r,\alpha)\\ &\leq& T\left(r,\frac{g_*^p(g_*-1)^qg_*^{\prime}}{\alpha}\right)+S(r,g_*) .\eeas Next applying Lemma \ref{l2.2} to the equation (\ref{e3.14}), we get \beas && T\left(r,\frac{g_*^p(g_*-1)^qg_*^{\prime}}{\alpha}\right)\\ &\leq& \ol N\left(r,\frac{g_*^p(g_*-1)^qg_*^{\prime}}{\alpha}\right)+ \ol N\left(r,\frac{\alpha}{g_*^p(g_*-1)^qg_*^{\prime}}\right)+\ol N(r,\mathcal{H})+S(r,g). \eeas \par So combining the above two we get, \bea\label{e3.15} T(r,g_*^p(g_*-1)^qg_*^{\prime})\leq \ol N\left(r,\frac{1}{g_*^p(g_*-1)^qg_*^{\prime}}\right)+2\ol N(r,g_*)+S(r,g_*). \eea \par By applying Lemmas \ref{l2.3}, \ref{l2.4} and (\ref{e3.15}), we have \beas && (p+q)T(r,g_*)\\ &\leq& T(r,g_*^p(g_*-1)^q)+S(r,g_*)\\ &\leq& T(r,g_*^p(g_*-1)^qg_*^{\prime})+T\left(r,\frac{1}{g_*^{\prime}}\right)+S(r,g_*)\\ &\leq& \ol N\left(r,\frac{1}{g_*^p(g_*-1)^qg_*^{\prime}}\right)+2\ol N(r,g_*)+T\left(r,\frac{1}{g_*^{\prime}}\right)+S(r,g_*)\\ &\leq& 8T(r,g_*)+S(r,g_*),  \eeas which contradicts $ p\geq q+9 $.\\
	\noindent{\bfseries{Subcase 2.}} If $ f_2=k $, where $ k $ is a constant.\\
	\noindent{\bfseries{Subcase 2.1}} If $ k\neq 1 $, then from the relation $ f_1+f_2+f_3=1 $, we get \bea\label{e3.16} \frac{f_*^p(f_*-1)^qf_*^{\prime}}{\alpha}-k\frac{g_*^p(g_*-1)^qg_*^{\prime}}{\alpha}=1-k.  \eea \par Next we apply Lemma \ref{l2.2} to the equation (\ref{e3.16}), we get \bea\label{e3.17} && T\left(r,\frac{f_*^p(f_*-1)^qf_*^{\prime}}{\alpha}\right)\\ &\leq& \ol N(r,g_*)+\ol N\left(r,\frac{1}{f_*^p(f_*-1)^qf_*^{\prime}}\right)++\ol N\left(r,\frac{1}{g_*^p(g_*-1)^qg_*^{\prime}}\right)+S(r,f_*) \nonumber. \eea
	By applying Lemma \ref{l2.3}, \ref{l2.4} and using equation (\ref{e3.17}), we get \beas && (p+q)T(r,f_*)\\ &=& T(r,f_*^p(f_*-1)^q)+S(r,f_*)\\ &\leq& T(r,f_*^p(f_*-1)^qf_*^{\prime})+T\left(r,\frac{1}{f_*^{\prime}}\right)+S(r,f_*)\\ &\leq& T\left(r,\frac{f_*^p(f_*-1)^qf_*^{\prime}}{\alpha}\right)+T\left(r,\frac{1}{f_*^{\prime}}\right)+S(r,f_*).  \eeas i.e., \beas (p+q-7)T(r,f_*)\leq 4T(r,g_*)+S(r,g_*).  \eeas\par Using Lemma \ref{l2.7}, we get \beas (p+q-4)T(r,f_*) \leq 4\left(\frac{p+q+2}{p-6}\right)T(r,f_*)+S(r,g_*),  \eeas which contradicts $ p\geq q+9 $.\\
	\noindent{\bfseries{Subcase 2.2}} Let $ k=1 $ i.e., $ \mathcal{H}=1 $ i.e., \beas  f_*^p(f_*-1)^qf_*^{\prime}\equiv g_*^p(g_*-1)^qg_*^{\prime}. \eeas On integrating, we get \beas \frac{f_*^{p+q+1}}{p+q+1}-\frac{^qC_1 f_*^{p+q}}{p+q}+\ldots+(-1)^q\frac{f_*^{p+1}}{p+1}\equiv \frac{g_*^{p+q+1}}{p+q+1}-\frac{^qC_1 g_*^{p+q}}{p+q}+\ldots+(-1)^q\frac{g_*^{p+1}}{p+1}+c, \eeas where $ c $ is an arbitrary constant. i.e., \bea \mathcal{F}\equiv\mathcal{G}+c.  \eea\par\noindent{\bfseries{Subcase 2.2.1}} Let if possible $ c\neq 0 $. Next we get \beas \Theta(0,\mathcal{F})+\Theta(c,\mathcal{F})+\Theta(\infty,\mathcal{F})=\Theta(0,\mathcal{F})+\Theta(0,\mathcal{G})+\Theta(\infty,\mathcal{F}).  \eeas We have, \beas \ol N\left(r,\frac{1}{\mathcal{F}}\right)=\ol N\left(r,\frac{1}{f_*}\right)+ \ol N\left(r,\frac{1}{f_*-a_1}\right)+\ldots+\ol N\left(r,\frac{1}{f_*-a_q}\right)\leq (q+1)\;T(r,f_*).\eeas \par Similarly, we get $ \ol\displaystyle N\left(r,\displaystyle\frac{1}{\mathcal{G}}\right)\leq (q+1)\;T(r,g_*).$\\ \par Again note that $ \ol N(r,\mathcal{F})=\ol N(r,f_*)\leq T(r,f_*). $ Again \beas T(r,\mathcal{F})=(p+q+1)\;T(r,f_*)+S(r,f_*).  \eeas
	\beas T(r,\mathcal{G})=(p+q+1)\;T(r,g_*)+S(r,g_*).  \eeas\par Thus \beas  \Theta(0,\mathcal{F})=1-\limsup_{r\rightarrow\infty}\frac{\ol N\left(r,\displaystyle\frac{1}{\mathcal{F}}\right)}{T(r,\mathcal{F})} \geq 1-\displaystyle\frac{(q+1)T(r,f_*)}{(p+q+1)T(r,f_*)}=\displaystyle\frac{p}{p+q+1}. \eeas\par Similarly \beas \Theta(0,\mathcal{H})\geq\frac{p}{p+q+1}\hspace{0.1in}\text{and}\hspace{0.1in}  \Theta(\infty,\mathcal{F})\geq\frac{p+q}{p+q+1}.\eeas \par Therefore \beas \Theta(0,\mathcal{F})+\Theta(c,\mathcal{F})+\Theta(\infty;\mathcal{F})\geq\frac{3p+q}{p+q+1}>2, \eeas since $ p\geq q+9 $, which is a contradiction.\par\noindent{\bfseries{Subcase 2.2.2}} Thus we get $ c=0 $. Thus we get \bea\label{e3.19} \mathcal{F}\equiv\mathcal{G}.  \eea\par Let $ h=\displaystyle\frac{f_*}{g_*} $. Then substituting in (\ref{e3.19}), we get \bea\label{e3.20}&& (p+q)(p+q-1)\ldots(p+1)g_*^q(h^{p+q-1}-1)\\&&\nonumber -\; ^qC_1(p+q+1)(p+q-1)\ldots(p+1)g_*^{q-1}(h^{p+q}-1)\\ &&\nonumber+\ldots+(-1)^q(p+q+1)(p+q)\ldots p(h^{p+1}-1)=0.  \eea\par 
	\noindent{\bfseries{Subcase 2.2.2.1.}} If $ h $ is a non-constant, then using Lemma \ref{lm9} and proceeding exactly same way as done in \cite[p-1272]{Wag & Shi-AM-2014}, we arrive at a contradiction.\par
	\noindent{\bfseries{Subcase 2.2.2.2.}} Let $ h $ is constant, then from (\ref{e3.20}), we get $ h^{p+q+1}-1=0 $, $ h^{p+q}-1=0 $, $ \ldots $, $ h^{p+1}-1=0 $. i.e., $ h^d-1=0 $, where $ d=gcd(p+q+1, p+q,\ldots, p+1)=1. $ i.e., $ h=1 $.\par Hence $ f_*\equiv g_*. $ i.e., $ f\equiv g $.\par
	\noindent{\bfseries{Subcase 3.}} Suppose $ f_3=c $, where $ c $ is a constant.\par
	\noindent{\bfseries{Subcase 3.1.}} If $ c\neq 1 $, then from the relation $ f_1+f_2+f_3=1 $, we get \bea\label{e3.21}\frac{f_*^p(f_*-1)^qf_*^{\prime}}{\alpha}-\frac{c\alpha}{g_*^p(g_*-1)^qg_*^{\prime}}=1-c.\eea \par Applying Lemma \ref{l2.2} to the above equation, we get \bea\label{e3.22} && T(r,f_*^p(f_*-1)^qf_*^{\prime})\\ &\leq&\nonumber T\left(r,\frac{f_*^p(f_*-1)^qf_*^{\prime}}{\alpha}\right)+S(r,f_*)\\ &\leq&\nonumber \ol N\left(r,\frac{f_*^p(f_*-1)^qf_*^{\prime}}{\alpha}\right)+\ol N\left(r,\frac{\alpha}{f_*^p(f_*-1)^qf_*^{\prime}}\right)+\ol N\left(r,\frac{g_*^p(g_*-1)^qg_*^{\prime}}{\alpha}\right)\\ &&\nonumber+S(r,f_*)\\ &\leq&\nonumber\ol N(r,f_*)+\ol N\left(r,\frac{1}{f_*^p(f_*-1)^qf_*^{\prime}}\right)+\ol N(r,g_*)+S(r,f_*).\eea Using Lemma \ref{l2.3}, \ref{l2.4} and (\ref{e3.22}), we have \beas &&(p+q)T(r,f_*)\\ &\leq& T(r,f_*^p(f_*-1)^q)+S(r,f_*)\\ &\leq& T\left(r,\frac{1}{f_*^p(f_*-1)^qf_*^{\prime}}\right)+ T\left(r,\frac{1}{f_*^{\prime}}\right)+S(r,f_*)\\ &\leq& 7\;T(r,f_*)+T(r,g_*)+S(r,f_*). \eeas\par Next by applying Lemma \ref{l2.7}, we get \beas && (p+q-7)\;T(r,f_*)\\ &\leq& T(r,g_*)+S(r,f_*)\\ &\leq& \left(\frac{p+q+2}{p-6}\right) T(r,f_*)+S(r,f_*),\eeas which contradicts $ p\geq q+9 $.\par 
	\noindent{\bfseries{Subcase 3.2.}} Let $ c=1 $. Then from (\ref{e3.21}), we get \bea\label{e3.23} f_*^p(f_*-1)^qf_*^{\prime}g_*^p(g_*-1)^qg_*^{\prime}=\alpha^2.\eea\par Let $ z_0 $ be a zero of $ f_* $ of order $ r_0 $. Then from (\ref{e3.23}), we see that $ z_0 $ is a pole of $ g_* $ of order $ s_0 $ (say). Then from (\ref{e3.23}), we get $  pr_0+r_0-1=ps_0+qs_0+s_0+1.$ i.e., $ (p+1)(r_0-s_0)=qs_0+2\geq p+1.$ i.e., \beas r_0\geq\frac{p+q+1}{q}.  \eeas \par Again let $ z_1 $ be a zero of $ f_*-1 $ of order $ r_1 $. Then from (\ref{e3.23}), we see that $ z_1 $ will be a pole of $ g_* $ of order $ s_1 $ (say). So we have $ r_1+r_1-1=ps_1+qs_1+s_1+1.$ i.e., \beas r_1\geq\frac{p+q+3}{2}.\eeas\par Let $ z_2 $ be a zero of $ f_*^{\prime} $ of order $ r_2 $ which are not the zero of $ f_*(f_*-1) $, so from (\ref{e3.23}) we see that $ z_2 $ will be a pole of $ g_* $ of order $ s_2 $ (say). Then from (\ref{e3.23}), we get $ r_2=ps_2+qs_2+s_2+1 $. i.e., \beas r_2\geq p+q+2.\eeas \par The similar explanations hold for the zeros of $ g_*^p(g_*-1)^qg_*^{\prime} $ also. Next we see from (\ref{e3.23}), we have \beas \ol N\left(r,f_*^p(f_*-1)^qf_*^{\prime}\right)=\ol N\left(r,\frac{\alpha^2}{g_*^p(g_*-1)^qg_*^{\prime}}\right).\eeas i.e., \beas && \ol N(r,f_*)\\ &\leq& \ol N\left(r,\frac{1}{g_*}\right)+ \ol N\left(r,\frac{1}{g_*-1}\right)+\ol N\left(r,\frac{1}{g_*^{\prime}}\right)\\ &\leq&\left(\frac{q}{p+q+1}\right)N\left(r,\frac{1}{g_*}\right)+\left(\frac{2}{p+q+3}\right)N\left(r,\frac{1}{g_*-1}\right)+\left(\frac{1}{p+q+2}\right)N\left(r,\frac{1}{g_*^{\prime}}\right)\\ &\leq&\left(\frac{q}{p+q+1}+\frac{2}{p+q+3}+\frac{2}{p+q+2}\right) T(r,g_*)+S(r,g_*). \eeas\par By applying Second Fundamental Theorem, we get \bea\label{e3.24} && T(r,f_*)\\ &\leq& \nonumber\ol N\left(r,\frac{1}{f_*}\right)+\ol N\left(r,\frac{1}{f_*-1}\right)+\ol N(r,f_*)+S(r,f_*)\\ &\leq&\nonumber \left(\frac{q}{p+q+1}+\frac{2}{p+q+3}\right)T(r,f_*)+\left(\frac{q}{p+q+1}+\frac{2}{p+q+3}+\frac{2}{p+q+2}\right) T(r,g_*)\\ &&\nonumber+S(r,f_*)+S(r,g_*).\eea\par Similarly, we get \bea\label{e3.25} && T(r,g_*)\\ &\leq&\nonumber \left(\frac{q}{p+q+1}+\frac{2}{p+q+3}\right)T(r,g_*)+\left(\frac{q}{p+q+1}+\frac{2}{p+q+3}+\frac{2}{p+q+2}\right) T(r,f_*)\\ &&\nonumber+S(r,f_*)+S(r,g_*). \eea \par From (\ref{e3.24}) and (\ref{e3.25}), we get \beas T_*(r)\leq \left(\frac{2q}{p+q+1}+\frac{4}{p+q+3}+\frac{2}{p+q+2}\right)T_*(r)+S_*(r).\eeas i.e., \beas  \left(1-\frac{2q}{p+q+1}-\frac{4}{p+q+3}-\frac{2}{p+q+2}\right)T_*(r)\leq S_*(r),\eeas which contradicts $ p\geq q+9 $.
\end{proof}
\begin{proof}[\bf{Proof of Theorem \ref{th.2}}] Since $ f_* $ and $ g_* $ both are non-constant entire functions, then we may consider the followings two cases.\par
\noindent{\bfseries{Case 1.}} Let $ f_* $ and $ g_* $ are two transcendental entire functions. Then it is clear that $ \ol N(r,f_*)=S(r,f_*) $ and $ \ol N(r,g_*)=S(r,g_*) $. With this the rest of the proof can be carried out in the line of the proof of Theorem \ref{th.1}. \par
\noindent{\bfseries{Case 2.}} Let $ f_* $ and $ g_* $ both are polynomials. Since $ f_*^p(f_*-1)^qf_*^{\prime} $ and $ g_*^p(g_*-1)^qg_*^{\prime} $ share $ \alpha $ $ CM $, then we must have \bea\label{e3.26}  ( f_*^p(f_*-1)^qf_*^{\prime} -\alpha)=\kappa\;( g_*^p(g_*-1)^qg_*^{\prime} -\alpha),\eea where $ \kappa $ is a non-zero constant.\par
\noindent{\bfseries{Subcase 2.1.}} Suppose $ \kappa\neq 1 $, then from (\ref{e3.26}), we get \bea\label{e3.27} \frac{f_*^p(f_*-1)^qf_*^{\prime}}{\alpha}-\kappa\frac{g_*^p(g_*-1)^qg_*^{\prime}}{\alpha}=1-\kappa. \eea\par Applying Lemma \ref{l2.2}, we get \bea\label{e3.28} &&  T(r,f_*^p(f_*-1)^qf_*^{\prime})\\ &\leq&\nonumber T\left(r,\frac{f_*^p(f_*-1)^qf_*^{\prime}}{\alpha}\right)+S(r,f_*)\\ &\leq&\nonumber \ol N\left(r,\frac{f_*^p(f_*-1)^qf_*^{\prime}}{\alpha}\right)+\ol N\left(r,\frac{\alpha}{f_*^p(f_*-1)^qf_*^{\prime}}\right)+\ol N\left(r,\frac{\alpha}{g_*^p(g_*-1)^qg_*^{\prime}}\right)\\ &&\nonumber+S(r,f_*)\\ &\leq& \ol N(r,f_*)+\ol N\left(r,\frac{\alpha}{f_*^p(f_*-1)^qf_*^{\prime}}\right)+\ol N\left(r,\frac{\alpha}{g_*^p(g_*-1)^qg_*^{\prime}}\right)+S(r,f_*)  \nonumber.\eea\par Using Lemmas \ref{l2.3}, \ref{l2.4} and (\ref{e3.27}), we get \beas && (p+q)T(r,f_*)\\ &\leq& T(r,f_*^p(f_*-1)^q)+S(r,f_*)\\ &\leq& T(r,f_*^p(f_*-1)^qf_*^{\prime})+T\left(r,\frac{1}{f_*}\right)+S(r,f_*^{\prime})\\ &\leq& 4\;T(r,f_*)+3\;T(r,g_*)+S(r,f_*).\eeas i.e., \beas (p+q-4)T(r,f_*)\leq 3\;T(r,g_*)+S(r,g_*).\eeas Using Lemma \ref{l2.8}, we get \beas (p+q-4)T(r,f_*)\leq 3\left(\frac{p+q+1}{p-3}\right)T(r,f_*)+S(r,f_*),\eeas which contradicts $ p\geq q+5 $.\\
\noindent{\bfseries{Subcase 2.2.}} Let $ \kappa=1 $. So from (\ref{e3.27}), we get \beas f_*^p(f_*-1)^qf_*^{\prime}\equiv g_*^p(g_*-1)^qg_*^{\prime}. \eeas\par Next proceeding exactly same way as done in \emph{Subcase 1.3.2}\; in the proof of \emph{Theorem \ref{th.1}}, we get $ f\equiv g $.
\end{proof}

\section{Concluding remarks and some open questions}
\par If we replace the condition $`` f_*^p(f_*-1)^qf_*^{\prime} $ and $ g_*^p(g_*-1)^qg_*^{\prime} $ share $ \alpha (z) $ $ CM $" by the condition $`` f_*^p(f_*-1)^qf_*^{\prime} $ and $ g_*^p(g_*-1)^qg_*^{\prime} $ share $ z $ $ CM $ ", then the conclusions of Theorems \ref{th.1} and \ref{th.2} still hold.\par Thus we get the following results
\begin{theo}\label{th.3}
	Let $ f $ and $ g $ hence $ f_*=f-w_p $ and $ g_*=g-w_p $, $ w_p\in\mathbb{C} $ be any two non-constant non- entire meromorphic functions, $ n\geq q+9 $, $ q $ $ \in\mathbb{N} $, be an integer. If $ \mathcal{P}_*(f)f_*^{\prime}=f_*^p(f_*-1)^qf_*^{\prime} $ and $ \mathcal{P}_*(g)g_*^{\prime}=g_*^p(g_*-1)^qf_*^{\prime} $ share $ z $  $ CM $, then $ f\equiv g $.
\end{theo}
\begin{theo}\label{th.4}
	Let $ f $ and $ g $ hence $ f_*=f-w_p $ and $ g_*=g-w_p $, $ w_p\in\mathbb{C} $ be any two non-constant entire functions, $ n\geq q+5 $, $ q $ $ \in\mathbb{N} $, be an integer. If $ \mathcal{P}_*(f)f_*^{\prime}=f_*^p(f_*-1)^qf_*^{\prime} $ and $ \mathcal{P}_*(g)g_*^{\prime}=g_*^p(g_*-1)^qf_*^{\prime} $ share $ z $  $ CM $, then $ f\equiv g $.
\end{theo}
\begin{note}\label{note4.1}
	If we choose $ q=m $, $ w_p=0 $, then since $ p=n+m-q $ and $ f_*=f-w_p $, so we get $ p=n $ and $ f_*=f $, respectively. With this we see that $ n\geq m+9 $ in Theorem \ref{th.3} and $ n\geq m+5 $ in Theorem \ref{th.4}.
\end{note}
\par So from the above note, we observe that Theorem \ref{th.3} and Theorem \ref{th.4} are the direct improvement as well as extension of Theorem H and I respectively.\par 
\begin{rem}
	We see from \emph{Note \ref{note4.1}} that for $ m=1 $ and $ m=2 $, we get $ n\geq 10 $ and $ n\geq 11 $ respectively in \emph{Theorem \ref{th.3}} which is a direct improvement of \emph{Theorem E} and \emph{F}.
\end{rem}
\begin{rem}
	For $ m=1 $, we see from \emph{Note \ref{note4.1}} that $ n\geq 6 $ in \emph{Theorem \ref{th.4}} which is a direct improvement of \emph{Theorem G}. 
\end{rem}\par Next for further research in this direction, one my glance over the following remarks.
\begin{rem}
	What worth noticing fact is that in  \cite[equation (39)]{Wag & Ana-AM-2016}, there is no term which is absent in the expression. So, for the case of $ h $ is constant, \cite[equation (40)]{Wag & Ana-AM-2016} implies $ h^d-1=0 $, where $ d=gcd(n+m+1, n+m,\ldots, n+1)=1. $ i.e., $ h=1 $ and hence $ f\equiv g $. But if we replace $ (f-1)^m $ in the expression $ f^n(f-1)^mf^{\prime} $ by a more general expression $ f^nP_m(f)f^{\prime} $, where $ P_m(f)=a_mf^m+a_{m-1}f^{m-1}+\ldots+a_1f+a_0 $, $ a_i\in\mathbb{C} $, for $ i=0,1,\ldots,m $. It is not always possible to handle the case of $ h $ is constant. If somehow one can do that, then from the case of $ h $ is constant, $ h^d-1=0 $, where $ d=gcd(n+m+1, n+m,\ldots,n+1)\neq 1 $ in general. So we can't obtained $ f\equiv g $ in general.
\end{rem}
\par Based on the above observations, we next pose the following open questions.
\begin{ques}
	Is it possible to reduce further the lower bounds of $ p $ in \emph{Theorem \ref{th.1}} and \emph{Theorem \ref{th.2}} ?
\end{ques}
\begin{ques}
	To get the uniqueness between $ f $ and $ g $ is it possible to replace $ f_*^p(f_*-1)^qf_*^{\prime} $ and $ g_*^p(g_*-1)^qg_*^{\prime} $ respectively by $ f_*^pP_m(f_*)f_*^{\prime} $ and $ g_*^pP_m(g_*)g_*^{\prime} $, where $ P_m(f_*)= P_m(f)=a_mf_*^m+a_{m-1}f_*^{m-1}+\ldots+a_1f_*+a_0  $ in \emph{Theorem \ref{th.1}} and \emph{Theorem \ref{th.2}} ?
\end{ques}

\end{document}